\documentclass[11pt]{article} 
\usepackage[utf8]{inputenc}

\usepackage{geometry}
\usepackage{graphicx}
\usepackage{tikz}

\usepackage{amsmath,amsfonts,amssymb}
\usepackage{amsthm}
\usepackage{xfrac, enumitem}
\usepackage{color}
\usepackage[pdftitle={On polynomial free-by-cyclic groups}, pdfauthor={Jean Pierre Mutanguha}, colorlinks, citecolor=blue, linkcolor=purple]{hyperref}
\usepackage[capitalise,nameinlink]{cleveref}

  \theoremstyle{plain}
\newtheorem{thm}{Theorem}[section]
\Crefname{thm}{Theorem}{Theorems}
\newtheorem*{thm*}{Theorem}
\newtheorem*{mthm*}{Main Theorem}
\newtheorem{prop}[thm]{Proposition}
\Crefname{prop}{Proposition}{Propositions}
\newtheorem{conj}[thm]{Conjecture}
\Crefname{conj}{Conjecture}{Conjectures}
\newtheorem{cor}[thm]{Corollary}
\Crefname{cor}{Corollary}{Corollaries}
\newtheorem*{mcor*}{Main Corollary}
\newtheorem{lem}[thm]{Lemma}

\newtheorem*{cor*}{Corollary}
\newtheorem*{obs*}{Observation}
\newtheorem*{sum*}{Summary}
\newtheorem*{claim*}{Claim}

  \theoremstyle{definition}

\newtheorem{qn}{Question}
\newtheorem*{qn*}{Question}
  \theoremstyle{remark}

\newtheorem*{conv}{Convention}

\newcommand{\mf}{\mathfrak}
\newcommand{\mb}{\mathbb}
\newcommand{\mc}{\mathcal}
\newcommand{\free}{\mb F}
\newcommand{\G}{\mb G}
\newcommand{\Z}{\mb Z}
\newcommand{\out}{\operatorname{Out}}
\newcommand{\Deg}{\operatorname{deg}}
\newcommand{\rank}{\operatorname{rank}}

\newcommand{\cay}{\operatorname{Cay}}

\newcommand{\fakeenv}{} 

\newenvironment{restate}[1]  
{
  \renewcommand{\fakeenv}{#1} 
  \theoremstyle{plain}
  \newtheorem*{\fakeenv}{\Cref{#1}} 
  \begin{\fakeenv}
}
{
  \end{\fakeenv}
}

\title{On polynomial free-by-cyclic groups}
\author{Jean Pierre Mutanguha\thanks{\emph{Email:} {\tt \href{mailto:jpmutang@princeton.edu}{jpmutang@princeton.edu}}, \emph{Web address:} {\tt \url{https://mutanguha.com}} \newline Department of Mathematics, Princeton University, Princeton, NJ, USA}}

\begin{document}
\maketitle

\begin{abstract}
A free-by-cyclic group can often be viewed as a mapping torus of a free group automorphism (monodromy) in multiple ways.
What dynamical properties must these monodromies share, and to what extent are they invariant under quasi-isometries?
We give a new proof using cyclic splittings that the growth type of a monodromy is a geometric invariant  of the free-by-cyclic group;
we also characterise the degree of polynomial growth using slender splittings.
For exponential growth, we conjecture that the nesting of attracting laminations is a geometric invariant.
\end{abstract}


\noindent These notes are expanded from a minicourse I taught at CRM-Montr\'eal in May 2023.
The minicourse consisted of three lectures that covered: dynamics of free group automorphisms (\cref{secPolyGrowth}); group invariance of growth type (\cref{secGrpInv}); and geometric invariance of growth type (\cref{secGeomInv,secRelHyp}).
We have included some new results, proofs of folklore statements, and new proofs of published results.


\section{Introduction}\label{secIntro}

A group is \underline{free-by-cyclic} if it has a free normal subgroup whose quotient group is cyclic;
we will insist the free subgroup is finitely generated and not trivial and the quotient is infinite.
In the last two sections, we will also assume the free subgroup is not cyclic to rule out, for convenience, the fundamental group of a torus or Klein bottle.

Rephrasing the definition, a free-by-cyclic group~$\G$ fits in a short exact sequence 
\[ 1 \to \free \to \G \to \Z \to 1, \text{ where } \free \text{ is finitely generated free and not trivial.} \]
Since~$\Z$ is free, the short exact sequence splits, and $\G \cong \free \rtimes_\Phi \Z$ for some automorphism $\Phi \colon \free \to \free$ that is well-defined up to composition with an inner automorphism of~$\free$.
There are of course plenty of things one can say about~$\G$;
we focus on how the dynamics of the outer automorphism $\phi = [\Phi]$ relate to the algebraic and geometric properties of~$\G$.

This is particularly interesting when~$\G$ has two finitely generated normal free subgroups $\free \neq \free'$ with cyclic quotients, and we get an isomorphism $\free \rtimes_\Phi \Z \cong \free' \rtimes_\Psi \Z$.
What dynamical properties must the outer automorphisms~$\phi$ and $\psi = [\Psi]$ have in common?

\begin{qn}
Generally, what must~$\phi, \psi$ have in common if $\G \cong \free \rtimes_\Phi \Z$ is quasi-isometric to $\G' \cong \free' \rtimes_\Psi \Z$?
(See \cref{secGeomInv} for a definition of quasi-isometry.)
\end{qn}

Our main result states that the two outer automorphisms must share growth type.

\begin{restate}{corQIInv}
If~$\G, \G'$ are quasi-isometric and~$\phi$ is polynomially growing, then so is~$\psi$.
\end{restate}

It is then a theorem of Nata\v{s}a Macura that they must also share their degrees of growth $\Deg(\psi) = \Deg(\phi)$ (\cref{thmDegQIInv}).
{Polynomial growth} and {degree} are defined in \cref{secPolyGrowth}.
\cref{corQIInv} was recently proven using relative hyperbolicity (\cref{thmRelHyp}), while our proof uses cyclic splittings.
Along the way, we give a very short proof that growth type is a group invariant (\cref{thmGroupInv}).
The previously known proof of group invariance through relative hyperbolicty was a bit overkill.

We have included characterisations and observations that may be known to experts but do not explicitly appear elsewhere;
here are three such observations.
\begin{restate}{lemIntersect}
If $\G' \le \G$ is a free-by-cyclic subgroup, then $\G' \cap \free$ is finitely generated.
\end{restate}

\noindent By  Feighn--Handel's coherence result \cite[Prop.\,2.3]{FH99}, we can replace ``free-by-cyclic subgroup'' with ``finitely generated noncyclic subgroup with Euler characteristic~0''.

\begin{restate}{lemUndistorted}
Free-by-cyclic subgroups of~$\G$ are undistorted.
\end{restate}
\begin{restate}{lemUndistorted2}
Slender subgroups of~$\G$ are undistorted.
\end{restate}

A finitely generated subgroup of~$\G$ is undistorted if the inclusion is a q.i.-embedding;
for instance, $\free$ is distorted in~$\G$ when~$\phi$ is exponentially growing or polynomially growing with degree at least 2.
A group is \underline{slender} if every subgroup is finitely generated;
for example, finitely generated abelian groups are slender.
Slender subgroups of~$\free$ are cyclic, and slender subgroups of~$\G$ are cyclic or $\Z \rtimes \Z$.
We also consider splittings of~$\G$ over slender subgroups and the {depths} of {slender hierarchies} for~$\G$ as defined in \cref{secGrpInv}.

\begin{restate}{lemZFreeSplit}
If~$\free$ is not cyclic, then slender (resp. cyclic) splittings of~$\G$ are naturally in one-to-one correspondence with $\phi$-fixed cyclic (resp. free) splittings of~$\free$.
\end{restate}

This correspondence allows us to give an algebraic characterisation of~$\Deg(\phi)$ when $\G$ is polynomial, i.e.~$\phi$ is polynomially growing.

\begin{restate}{corGroupInv}
$\Deg(\phi)$ is the minimal depth of slender hierarchies for polynomial~$\G$.
\end{restate}

With the hope of turning this characterisation into a geometric invariant, we guess that slender splittings of~$\G$ are preserved by quasi-isometries.

\begin{restate}{conjSlender}
If $\G, \G'$ are quasi-isometric, then a slender splitting of~$\G$ induces a slender splitting of~$\G'$.
\end{restate} 

For cyclic splittings, the conjecture is a theorem of Panos Papasoglu (\cref{thmZSplit});
this is the key to our proof of \cref{corQIInv}.
When~$\G$ is also the fundamental group of a 3-manifold, then the conjecture is a theorem of Kapovich--Leeb \cite[Thm.\,1.1]{KL97}.

Using~\cref{lemIntersect} and a proposition of Gilbert Levitt (\cref{propPolySubgp}), we give an algebraic characterisation of the polynomial part of a free-by-cyclic group.
The characterisation also follows from \cref{lemUndistorted} and relative hyperbolicity.

\begin{restate}{propPolySubgp2}
There is a unique (up to conjugacy) finite set~$\mc P(\G)$ 
of pairwise nonconjugate maximal polynomial free-by-cyclic subgroups of~$\G$
\end{restate}

Bestvina--Feighn--Handel \cite[\S3]{BFHTits1} define the partially ordered finite set~$\Lambda^+(\phi)$ of attracting laminations for~$\phi$. 
Define the \underline{lamination depth} $\mf d(\phi)$ to be the size of the longest chains in~$\Lambda^+(\phi)$. So~$\phi$ is polynomially growing if and only if $\mf d(\phi) = 0$.

\begin{restate}{conjHeight}
If $\G, \G'$ are quasi-isometric, then $\mf d(\phi) = \mf d(\psi)$.
\end{restate}

It is open whether the lamination depth is even a group invariant!
This conjecture is closely related to our previous conjecture that being both fully irreducible and atoroidal was a geometric invariant -- see the discussion at the end of \cref{secRelHyp}.

\medskip
\noindent \textbf{Acknowledgements.}
I thank: Genevieve Walsh for organizing the workshop/conference and inviting me to give a minicourse;
and CRM-Montr\'eal for its hospitality.

\section{Polynomial growth}\label{secPolyGrowth}
Throughout,~$\free$ denotes a finitely generated nontrivial free group, $\Phi \colon \free \to \free$ an automorphism, and $\phi = [\Phi]$ its outer class.

A \underline{simplicial tree} is a simply connected 1-dimensional cell complex (equivalently, a contractible graph).
A \underline{cyclic splitting} of a group is a simplicial tree with a minimal nontrivial simplicial action of the group whose edge (setwise) stabilisers are cyclic;
it is a \underline{free splitting} if the edge stabilisers are trivial.
Note that we allow the group to act with inversion on edges.
A free splitting of~$\free$ is \underline{absolute} if vertex stabilisers are trivial -- this means the action is free, but we introduce ``absolute'' to avoid the phrase ``free free splittings'' later on.

Let~$T$ be an absolute free splitting of~$\free$.
For $x \in \free$, the \underline{$T$-length} $\|x\|_T$ is the minimal distance~$x$ translates a point in~$T$.
A conjugacy class $[x]$ \underline{grows exponentially} on (forward) $\phi$-iteration if
$\underset{n \to \infty}{\liminf}\, \|\Phi^n(x)\|_T^{1/n} > 1$;
it \underline{grows polynomially} on $\phi$-iteration if the sequence $(\|\Phi^n(x)\|_T)_{n=1}^\infty$ is bounded above by a polynomial in~$n$ -- the smallest degree of such a polynomial is the \underline{degree} for~$[x]$.
These properties (grows exp./poly. \& lim-inf/degree) are mutually exclusive and independent of the absolute free splitting~$T$.
For the interested reader, Levitt \cite[\S6]{Lev09} gives a finer classification of growth rates.
The outer automorphism~$\phi$ is \underline{exponentially growing} if some conjugacy class~$[x]$ grows exponentially on $\phi$-iteration;
it is \underline{polynomially growing} if every conjugacy class~$[x]$ grows polynomially on $\phi$-iteration.

A map $g \colon T \to T$ on a cyclic splitting of $\free$ is \underline{$\Phi$-equivariant} if $g(x \cdot p) = \Phi(x) \cdot g(p)$ for all $x \in \free, p \in T$.
A cyclic splitting~$T$ of~$\free$ is \underline{$\phi$-fixed} if it admits a $\Phi$-equivariant simplicial automorphism;
such a simplicial automorphism is unique when~$\free$ is not cyclic.
Marc Culler \cite[Thm.\,2.1]{Cul84} proved our first characterisation of a dynamical property of~$\phi$ in terms of absolute free splittings:

\begin{thm}\label{thmNRP}
$\free$ has a $\phi$-fixed absolute free splitting if and only if $\phi$ has finite order. \qed
\end{thm}

By solving the Nielsen realisation problem for~$\out(\free)$, Culler actually proved a stronger theorem: $\mf G \le \out(\free)$ is finite if and only if~$\free$ has a {$\mf G$-fixed} absolute free splitting.
Bestvina--Handel \cite[\S1]{BH92} generalised \Cref{thmNRP} to polynomial growth using the next crucial observation:

\begin{thm}\label{thmFixedFreeSplitting}
If~$\free$ has no $\phi$-fixed free splitting, then~$\phi$ is exponentially growing.
\end{thm}

\begin{proof}[Sketch]
Bestvina--Handel \cite[Thm.\,1.7]{BH92} prove that an \emph{irreducible} outer automorphism of~$\free$ is represented by an irreducible \emph{train track map};
the map is either \emph{expanding} or a homeomorphism.
They \cite[Rmk.\,1.8]{BH92} also prove that an outer automorphism is exponentially growing if it is represented by an expanding irreducible train track map.

While they worked with connected finite graphs~$\Gamma$ with $\pi_1(G) \cong \free$, their argument works almost verbatim for free splittings of~$\free$ with $\phi$-invariant vertex stabilisers.
In particular, some free splitting of~$\free$ admits a $\Phi$-equivariant irreducible train track map;
the map is either expanding or a simplicial automorphism.
As $\free$ has no $\phi$-fixed free splitting, the train track map is expanding, and the outer automorphism~$\phi$ is exponentially growing.
\end{proof}

Let~$T$ be a cyclic splitting of~$\free$.
By Bass--Serre theory \cite[I-\S5.4]{Ser77}, there are finitely many orbits $\free \cdot v_i$ in~$T$ of vertices with nontrivial stabilisers. 
The conjugacy classes~$[\mb S_i]$ of nontrivial vertex stabilisers are represented by finitely generated proper subgroups of~$\free$;
the \underline{children} of~$T$ are the representatives~$\mb S_i$.
Any $\Phi$-equivariant simplicial automorphism $g\colon T \to T$ permutes the orbits $\free \cdot v_i$ of vertices with nontrivial stabilisers, the conjugacy classes $[\mb S_i]$ of the children are $\phi$-periodic, and the \underline{restrictions} $\phi_i \in \out(\mb S_i)$ are the outer automorphisms induced by the first-return maps $g_i \colon \free \cdot v_i \to \free \cdot v_i$ under~$g$.
The outer automorphism~$\phi$ is exponentially growing if some restriction~$\phi_i$ is exponentially growing.

Now assume~$T$ is a $\phi$-fixed free splitting.
So the children are proper free factors~$\free_i$.
Since they have smaller rank than~$\free$, we can induct on rank.
By inductively considering $\phi_i$-fixed free splittings of the children~$\free_i$, we get a \underline{$\phi$-fixed free hierarchy} for~$\free$: a {family tree} starting with~$\free$ such that each terminal {descendant}~$\free'$ has a $\psi$-fixed absolute free splitting or no $\psi$-fixed free splitting, where $\psi \in \out(\free')$ is the restriction of~$\phi$.
The \underline{depth} of a hierarchy is the length of the longest branches in the family tree. 
The \underline{free depth} $\delta(\phi)$ of~$\phi$ is the minimal depth of $\phi$-fixed free hierarchies for~$\free$;
for instance, $\delta(\phi) = 0$ means~$\free$ has a $\phi$-fixed absolute free splitting or no $\phi$-fixed free splitting.

Similarly, a \underline{$\phi$-fixed cyclic hierarchy} for~$\free$ is a family tree that starts with~$\free$, consists of descendants that are children of fixed cyclic splittings, and whose terminal descendants have fixed absolute free splittings or no fixed cyclic splittings.
A cyclic hierarchy for~$\free$ is \underline{complete} if its terminal descendants have fixed absolute free splittings.
The \underline{cyclic depth} $\delta_c(\phi)$ of~$\phi$ is the minimal depth of $\phi$-fixed cyclic hierarchies for~$\free$;
by definition, $\delta_c(\phi) \le \delta(\phi)$.

We use complete cyclic hierarchies for~$\free$ to prove the dichotomy of growth type for~$\phi$:

\begin{prop}\label{propDichotomy}
The following are equivalent:
\begin{enumerate}
\item every $\phi$-fixed free hierarchy for~$\free$ is complete;
\item some $\phi$-fixed cyclic hierarchy for~$\free$ is complete;
\item $\phi$ is polynomially growing; and
\item $\phi$ is not exponentially growing.
\end{enumerate}
\end{prop}
\begin{proof}[Sketch]
$1{\Rightarrow}2$:
$\free$ has at least one $\phi$-fixed free hierarchy.

$2{\Rightarrow}3$:
Suppose~$\free$ has a complete $\phi$-fixed cyclic hierarchy with depth~$d$, and let $g \colon T \to T$ be the initial $\Phi$-equivariant simplicial automorphism of a cyclic splitting that produced the hierarchy.
Without loss of generality, assume~$\free$ acts on~$T$ without inversion, and~$g$ transitively permutes the $\free$-orbits of edges in~$T$.
For the base case,~$T$ is absolute. 
So, the length of every edge-path in~$T$ is preserved on $g$-iteration.
In the quotient graph $\Gamma = \free \backslash T$, every edge-path grows polynomially on $\overline g$-iteration with degree $0 = d$, where $\overline g \colon \Gamma \to \Gamma$ is the simplicial automorphism induced by~$g$.

For induction, assume~$\free_i$ are the children of~$T$, $\Gamma_i$ are connected finite graphs with $\pi_1(\Gamma_i) \cong \free_i$, $\overline{g_i} \colon \Gamma_i \to \Gamma_i$ are cellular maps representing~$\phi_i$, and every edge-path in~$\Gamma_i$ grows polynomially \emph{(rel.~endpoints)} with degree $< d$ on $\overline{g_i}$-iteration.
We focus on the case~$T$ is not a free splitting -- the free splitting case is similar (and simpler).
Let~$Z$ be a ``blow up'' of~$\free \backslash T$ rel.~$\Gamma_i$:
for each vertex~$\overline{v_i}$ in~$\Gamma$, we have the graph~$\Gamma_i$;
and for each edge~$\overline{e_{ij}}$ in~$\Gamma$, we have a cylinder $\mb S^1 \times [0,1]$ whose boundary components are attached to $\Gamma_i$ and $\Gamma_j$ along loops representing the conjugacy classes in~$\free_i$ and~$\free_j$ corresponding to the stabiliser of the edge $e_{ij}$ in~$T$.
So,~$Z$ is a connected finite 2-complex with $\pi_1(Z) \cong \free$.
The self-maps~$g$ and~$\overline{g_i}$ induce a cellular map $h \colon Z \to Z$ that represents~$\phi$, and its restrictions to the cylinders are homotopic to \emph{Dehn twists};
therefore, every edge-path in the 1-skeleton~$Z^{(1)}$ grows polynomially with degree $\le d$ on $h$-iteration.
Consequently, if $\Gamma'$ is a graph with $\pi_1(\Gamma') \cong \free$ and $g' \colon \Gamma' \to \Gamma'$ is cellular map representing~$\phi$, then every edge-path in~$\Gamma'$ grows polynomially with degree $\le d$ on $g'$-iteration.
By induction, the conjugacy class~$[x]$ grows polynomially with degree $\le d$ on $\phi$-iteration for all $x \in \free$.

$3{\Rightarrow}4$: Polynomial functions are subexponential.

$4{\Rightarrow}1$: Suppose some $\phi$-fixed free hierarchy for~$\free$ is not complete, i.e.~some terminal descendant $\free' \le \free$ has no $\psi$-fixed free splitting, where $\psi \in \out(\free')$ is the restriction of~$\phi$.
So~$\psi$, and hence~$\phi$, is exponentially growing by \cref{thmFixedFreeSplitting}.
\end{proof}

Let~$\phi$ be polynomially growing.
The proof also shows that the degree for~$[x]$ is at most~$\delta_c(\phi)$ for all $x \in \free$;
define the \underline{degree} $\Deg(\phi)$ of~$\phi$ to be the maximum such degree.
A $\phi$-fixed cyclic splitting of~$\free$ is \underline{simplifying} if it determines restrictions~$\phi_i$ with $\Deg(\phi_i) < \Deg(\phi)$.
The next strengthening of \cref{thmFixedFreeSplitting} is due to Bestvina--Feighn--Handel \cite[\S4]{BFHTits2}:

\begin{thm}\label{thmFixedDegSplitting}
$\free$ has a simplifying $\phi$-fixed cyclic splitting when~$\phi$ is polynomially growing.
\end{thm}

\begin{proof}[Sketch]
If $\Deg(\phi) = 0$, then~$\phi$ has finite order and,
by \cref{thmNRP}, $\free$ has a $\phi$-fixed absolute free splitting.
If $\Deg(\phi) = 1$, then the $\phi$-fixed cyclic splitting in \cite[Lem.\,4.37]{BFHTits2} determines finite order restrictions~$\phi_i$.
Finally, if $\Deg(\phi) \ge 2$, then the $\phi$-fixed free splitting in \cite[Lem.\,4.33]{BFHTits2} determines restrictions~$\phi_i$ with $\Deg(\phi_i) < \Deg(\phi)$.
\end{proof}



As a corollary, growth type and degree are invariant under inverses:

\begin{cor}\label{corPolyInverse}
Generally, $\delta(\phi^{-1}) = \delta(\phi)$ and $\delta_c(\phi^{-1}) = \delta_c(\phi)$.

If~$\phi$ is polynomially growing, then so is~$\phi^{-1}$, and $\Deg(\phi^{-1}) = \Deg(\phi) = \delta_c(\phi)$.
\end{cor}
\begin{proof}
Let~$T$ be a free splitting of~$\free$ and $g \colon T \to T$ a $\Phi$-equivariant simplicial automorphism.
Then the inverse $g^{-1} \colon T \to T$ is a $\Phi^{-1}$-equivariant simplicial automorphism;
thus~$T$ is also $\phi^{-1}$-fixed.
So $\phi$-fixed free hierarchies for~$\free$ are exactly the $\phi^{-1}$-fixed free hierarchies, and $\delta(\phi^{-1}) = \delta(\phi)$.
Similarly, $\phi$-fixed cyclic hierarchies for~$\free$ are exactly the $\phi^{-1}$-fixed cyclic hierarchies, and $\delta_c(\phi^{-1}) = \delta_c(\phi)$.

In particular,  if $\phi$ is polynomially growing, so is $\phi^{-1}$ by \cref{propDichotomy}($2{\Leftrightarrow}3$).
We already know $\Deg(\phi) \le \delta_c(\phi)$.
By inductively invoking \cref{thmFixedDegSplitting}, we get a $\phi$-fixed cyclic hierarchy for~$\free$ whose depth is at most~$\Deg(\phi)$.
So $\Deg(\phi) = \delta_c(\phi)$.
\end{proof}

When $\phi$ is exponentially growing, Levitt \cite[Prop.\,1.4]{Lev09} remarkably showed that the polynomially growing conjugacy classes are supported in a unique subgroup system:

\begin{prop}\label{propPolySubgp}
There is a finite set~$\mc P(\phi)$ of pairwise nonconjugate nontrivial subgroups of~$\free$ satisfying the following for any nontrivial $x \in \free$:
\begin{quote}
if~$x$ is conjugate to an element of a subgroup in~$\mc P(\phi)$, then~$[x]$ grows polynomially on $\phi$-iteration;
otherwise,~$[x]$ grows exponentially on $\phi$-iteration.
\end{quote}

Each $\mb P \in \mc P(\phi)$ is finitely generated. 
{Essentially distinct} conjugates of $\mb P_1, \mb P_2 \in \mc P(\phi)$ have trivial intersection.
Any nontrivial subgroup of~$\free$ consisting entirely of elements whose conjugacy classes in~$\free$ grow polynomially on $\phi$-iteration is conjugate to a subgroup of some $\mb P \in \mc P(\phi)$;
therefore, the set~$\mc P(\phi)$ is unique (up to conjugacy).
\end{prop}

\noindent Conjugates $g_1 \mb P_1 g_1^{-1}, g_2 \mb P_2 g_2^{-1}$ are \underline{essentially distinct} if $\mb P_1 \neq \mb P_2$ or $g_2^{-1}g_1 \notin \mb P_1 = \mb P_2$.
The set~$\mc P(\phi)$ can be empty, and $\mc P(\phi) = \{ \free \}$ if and only if~$\phi$ is polynomially growing.

\begin{proof}[Sketch]
Suppose~$\phi$ is exponentially growing.
A sort of converse to \cref{thmFixedFreeSplitting} states that there is a minimal nontrivial isometric $\mb F$-action on an $\mb R$-tree~$T$ with trivial arc stabilisers, and the $\mb R$-tree $T$ admits a $\Phi$-equivariant expanding homothety.
By Gaboriau--Levitt's index theory \cite[Thm.\,III.2]{GL95}, nontrivial point stabilisers have smaller rank than $\free$ and are partitioned into finitely many conjugacy classes permuted by~$\phi$;
moreover, if $x \in \free$ does not fix a point in~$T$, then~$[x]$ grows exponentially on $\phi$-iteration.
As the $\free$-action on~$T$ has trivial arc stabilisers, nontrivial point stabilisers of distinct points will have trivial intersection.

Each child~$\mb S_i$ of~$T$ (i.e.~a nontrivial point stabiliser representing a conjugacy class) has the restriction $\phi_i \in \out(\mb S_i)$.
By inductively considering children with exponentially growing restrictions, we get a finite family tree whose terminal descendants~$\mb P \in \mc P(\phi)$ do not have exponentially growing restrictions.
So the restriction~$\phi_i$ to each $\mb P_i \in \mc P(\phi)$ is polynomially growing by \cref{propDichotomy}($4{\Rightarrow}3$), and the conjugacy classes in~$\mb P_i$ grow polynomially on $\phi_i$-iteration.
In particular, the conjugacy class $[x]$ (in~$\free$) grows polynomially on $\phi$-iteration if $x$ is conjugate to an element in $\mb P_i$.

For uniqueness of~$\mc P(\phi)$, suppose a nontrivial subgroup $\mb P' \le \free$ consists entirely of elements whose conjugacy classes in~$\free$ grow polynomially on $\phi$-iteration.
By construction of the family tree,~$\mb P'$ is conjugate to a subgroup of some unique terminal descendant. 
\end{proof}

As a corollary, we get the growth type dichotomy for conjugacy classes:~$[x]$ grows polynomially on $\phi$-iteration if and only if it does not grow exponentially on $\phi$-iteration.
By uniqueness, the outer automorphism~$\phi$ permutes the conjugacy classes of the subgroups in~$\mc P(\phi)$.
Each $\mb P_i \in \mc P(\phi)$ has the polynomially growing restriction $\phi_i \in \out(\mb P_i)$ induced by the first-return map under~$\phi$.
By \cref{corPolyInverse} and uniqueness, 
$\mc P(\phi^{-1}) = \mc P(\phi)$.

We end the section by showing that the growth type and degree are inherited by proper powers and restrictions to finite index subgroups.
Suppose $\free' \le \free$ is a finite index subgroup.
Then $\Psi = \Phi^k|\free'$ is an automorphism of~$\free'$ for some $k \ge 1$.
Let $\psi = [\Psi]$ be the outer automorphism in $\out(\free')$.

\begin{lem}\label{lemVirtual}
For $n \ge 1$, $\phi^n$ is polynomially growing if and only if~$\phi$ is polynomially growing.
The restriction $\psi$ is polynomially growing if and only if~$\phi$ is polynomially growing;
moreover, $\deg(\psi) = \deg(\phi) = \deg(\phi^n)$ for $n \ge 1$ when $\phi$ is polynomially growing.
\end{lem}

\noindent As a corollary, $\mc P(\phi^n) = \mc P(\phi)$ for all $n \neq 0$.

\begin{proof}
For $x \in \free$, if the conjugacy class~$[x]$ grows polynomially growing on $\phi$-iteration with degree~$d$, then it is immediate from the definition that~$[x]$ grows polynomially on $\phi^n$-iteration with degree~$d$ for $n \ge 1$.
So, for $n \ge 1$,~$\phi^n$ is polynomially growing with $\Deg(\phi^n) = \Deg(\phi)$ if~$\phi$ is polynomially growing.
Similarly, if~$[x]$ grows exponentially on $\phi$-iteration, then it also grows exponentially on $\phi^n$-iteration for $n \ge 1$;
thus~$\phi^n$ is exponentially growing if~$\phi$ is exponentially growing.

Pick an absolute free splitting~$T$ of~$\free$;
restricting this action to~$\free'$ gives us an absolute free splitting $T'$ of~$\free'$.
Then $\|x'\|_{T'} = \|x'\|_T$ for $x' \in \free'$, and the conjugacy class~$[x']_{\free'}$ (in~$\free'$) grows polynomially on $\psi$-iteration with degree~$d$ if and only if the conjugacy class~$[x']_\free$ (in~$\free$) grows polynomially on $\phi^k$-iteration with degree~$d$.
For this reason,~$\phi^k$ is exponentially growing if~$\psi$ is exponentially growing.
Conversely, if~$\psi$ is polynomially growing and $x \in \free$, then $x^m \in \free'$ for some $m\ge 1$, and the conjugacy class~$[x^m]_\free$ (and hence~$[x]_\free$) grows polynomially on $\phi^k$-iteration with the same degree as $[x^m]_{\free'}$ grows on~$\psi$-iteration;
therefore,~$\phi^k$ is polynomially growing, and $\Deg(\phi^k) = \Deg(\psi)$.
\end{proof}

\section{Group invariance}\label{secGrpInv}
Thoughout,~$\G$ denotes the free-by-cyclic group $\free \rtimes_\Phi \Z$. 
Our goal is to characterise the growth type of~$\phi$ as an algebraic property of~$\G$.


\begin{conv}$\G$ has the relative presentation $\langle \free, t~|~txt^{-1} = \Phi(x) \text{ for all } x \in \free \rangle$.\end{conv}

A cyclic splitting of~$\G$ is \underline{absolute} if its vertex stabilisers are cyclic.
A \underline{$\Z$-splitting} of a group is a cyclic splitting whose edge stabilisers are infinite;
we say a group ``splits over~$\Z$'' if it has a $\Z$-splitting.
The \underline{children} of a splitting of~$\G$ are the representatives of conjugacy classes of noncyclic vertex stabilisers.
A \underline{$\Z$-hierarchy} for~$\G$ is a family tree that starts with~$\G$, consists of descendants that are children of $\Z$-splittings, and terminates on descendants that have absolute $\Z$-splittings or no $\Z$-splittings.
The \underline{$\Z$-depth} $\delta_\Z(\G)$ of~$\G$ is the minimal depth of $\Z$-hierarchies for~$\G$.

A \underline{slender splitting} of a group is a simplicial tree with a minimal nontrivial simplicial action of the group whose edge stabilisers are slender.
A \underline{slender hierarchy} for~$\G$ is a family tree that starts with~$\G$, consists of descendants that are children of slender splittings, and terminates on descendants that have absolute cyclic splittings or no slender splittings.
A slender hierarchy for~$\G$ is \underline{complete} if its terminal descendants have absolute cyclic splittings.
The \underline{slender depth} $\delta_s(\G)$ of~$\G$ is the minimal depth of slender hierarchies for~$\G$.
As defined, it is not clear that $\Z$-depth and slender depth for~$\G$ are finite.

\begin{lem}\label{lemZFreeSplit}
If~$\free$ is not cyclic, then slender (resp. cyclic) splittings of~$\G$ are naturally in one-to-one correspondence with $\phi$-fixed cyclic (resp. free) splittings of~$\free$.
\end{lem}

\noindent A variation of this appears in the proof of \cite[Cor.\,15]{KK00}.
The same variation is also proven in the second paragraph of \cite[\S1]{Bri02}.

\begin{proof}
Suppose~$T$ is a $\phi$-fixed cyclic splitting of~$\free$.
As~$\free$ is not cyclic, a unique simplicial automorphism $g \colon T \to T$ is $\Phi$-equivariant, which is precisely the property needed to extend the $\free$-action on~$T$ to a $\G$-action by $t \cdot p = g(p)$ for $p \in T$ -- this is the unique $\G$-action extending the $\free$-action on~$T$.
For any edge~$e$ of~$T$ with $\free$-stabiliser $\langle y \rangle$, an iterate of~$g$ fixes the orbit $\free \cdot e$, and the $\G$-stabiliser of~$e$ is $\langle y, xt^n\rangle$ for some $x\in \free$.
If~$y$ is not trivial, then $\langle y, xt^n \rangle \cong \langle y \rangle \rtimes \langle xt^n \rangle$ by $\Phi$-equivariance of~$g$;
otherwise, $\langle y, xt^n \rangle \cong \Z$.
So~$T$ is a slender splitting of~$\G$, and it is a $\Z$-splitting of~$\G$ if~$T$ was a free splitting of~$\free$.

Conversely, suppose~$T$ is a slender splitting of~$\G$.
As~$\free$ is not cyclic, $\free \trianglelefteq \G$ acts nontrivially and minimally on~$T$, and the action on~$T$ by~$t \in \G$ is $\Phi$-equivariant.
So~$T$ is a $\phi$-fixed slender (hence cyclic) splitting of~$\free$.
Now suppose~$T$ was a cyclic splitting of~$\G$.
Since~$\free$ is finitely generated,~$T$ has finitely many $\free$-orbits of edges.
For any edge~$e$ of~$T$ with a nontrivial $\free$-stabiliser $\langle y \rangle$, a power~$t^n$ fixes the orbit $\free \cdot e$, and the $\G$-stabiliser of~$e$ has the noncyclic subgroup $\langle y, xt^n\rangle$ for some $x \in \free$ and $n \ge 1$ -- yet~$T$ is a cyclic splitting of~$\G$!
So~$T$ is a $\phi$-fixed free splitting of~$\free$.
\end{proof}

Suppose~$\free$ is cyclic, i.e.~$\G \cong \Z \rtimes \Z$.
A noncyclic torsion-free group with a free splitting will contain a noncyclic free subgroup.
As~$\G$ is virtually abelian, it has no free splittings.
So any cyclic splitting of~$\G$ is a $\Z$-splitting.
Since~$\free$ is cyclic, it has a unique cyclic splitting, and this is an absolute free splitting;
however,~$\G$ has infinitely many (resp. exactly two) slender splittings if $\Phi \colon \free \to \free$ is the identity (resp. the involution), and they are all absolute $\Z$-splittings.
In this case,~$\G$ has a unique $\Z$-hierarchy, and this hierarchy is complete.


Now suppose $\free$ is not cyclic.
It follows from the proof of \cref{lemZFreeSplit} that any slender splitting of~$\G$ has infinite edge stabilisers, and its children are free-by-cyclic subgroups.
Thus, by \cref{lemZFreeSplit}, slender hierarchies (resp. $\Z$-hierarchies) for~$\G$ are naturally in one-to-one correspondence with $\phi$-fixed cyclic (resp. free) hierarchies for~$\free$.
In particular, every $\Z$-hierarchy for~$\G$ has finite depth, and a slender hierarchy for~$\G$ is complete exactly when it corresponds to a complete $\phi$-fixed cyclic hierarchy for~$\free$.
With the previous paragraph, \cref{propDichotomy} translates into the group invariance of growth type:

\begin{thm}\label{thmGroupInv}
The following are equivalent:
\begin{enumerate}
\item every $\Z$-hierarchy for~$\G$ is complete;
\item some slender hierarchy for~$\G$ is complete; and
\item $\phi$ is polynomially growing. \qed
\end{enumerate}
\end{thm}

$\G$ is \underline{polynomial} if the conditions of the theorem hold.
The one-to-one correspondences of hierarchies imply $\delta_\Z(\G) = \delta(\phi)$ and $\delta_s(\G) = \delta_c(\phi)$.
By \cref{corPolyInverse}, we get an algebraic characterisation of the degree:

\begin{cor}\label{corGroupInv}
$\delta_s(\G) = \Deg(\phi)$ when~$\G$ is polynomial. \qed
\end{cor}

\noindent This characterisation of~$\Deg(\phi)$ may be known to experts but is not stated in the literature.

Suppose $\G' \le \G$ is a finite index subgroup.
Then $\free' = \G' \cap \free$ has finite index in~$\free$; 
in particular, it is finitely generated.
As~$\G'$ is not a subgroup of~$\free$, it is generated by~$\free'$ and~$xt^n$ for some $x \in \free$ and $n \ge 1$.
So $\G' = \free' \rtimes_\Psi \Z$ is a free-by-cyclic group, where $\Psi \colon \free' \to \free'$ be given by $x' \mapsto x\Phi^n(x')x^{-1}$.
An immediate consequence of \cref{lemVirtual} and group invariance of the degree (\cref{corGroupInv}) is {commensurability} invariance:

\begin{cor}\label{corVirtualPoly}
Let~$\G' \le \G$ be a finite index subgroup.
$\G'$ is polynomial if and only if~$\G$ is polynomial;
$\delta_s(\G') = \delta_s(\G)$ when~$\G$ is polynomial. \qed
\end{cor}

\noindent The identities $\delta_\Z(\G) = \delta_\Z(\G')$ and $\delta_s(\G') = \delta_s(\G)$ always hold -- the proof we have in mind invokes the algebraic torus theorem \cite{DS00} (see also \cref{conjSlender,corZHierarchy}).


We now turn the subgroup system~$\mb P(\phi)$ into a group invariant of~$\G$.
Each $\mb P_i \in \mc P(\phi)$ representing a $\phi$-orbit of conjugacy classes determines a polynomial free-by-cyclic subgroup $\G_i = \mb P_i \rtimes_{\Phi_i} \Z$ of~$\G$, where $\Phi_i \colon \mb P_i \to \mb P_i$ represents the polynomially growing restriction~$\phi_i$;
denote the set of subgroups~$\G_i$ by~$\mc P(\G)$.
As before, $\mc P(\G)$ can be empty, and $\mc P(\G) = \{ \G \}$ if and only if~$\G$ is polynomial.
Our goal is to give an algebraic characterisation of~$\mc P(\G)$.

First, here is a remarkable property of free-by-cyclic subgroups in~$\G$:

\begin{lem}\label{lemIntersect}
If $\G' \le \G$ is a free-by-cyclic subgroup, then $\G' \cap \free$ is finitely generated.
\end{lem}

\noindent This lemma was the key idea in \cite[Thm.\,4.3]{Mut21}.
It is indispensable for turning algebraic statements about~$\G$ into dynamical statements about~$\phi$.
Despite its importance, it is not widely known.
A minor variation is stated in the paragraph following \cite[Thm.\,A]{HW10}.

\begin{proof}
Let $\pi \colon \G \to \Z$ be the homomorphism that maps $\free \mapsto 0$ and $t \mapsto 1$.
Then $\free' = \G' \cap \free$ is the kernel of the restriction $\pi|\G'$.
By our definition of free-by-cyclic groups, $\G'$ is a finitely generated noncyclic subgroup with Euler characteristic $\chi(\G') = 0$;
thus~$\G'$ is generated by~$\free'$ and $s = xt^n$ for some $x \in \free$ and $n \ge 1$.
In proving the coherence of $\G = \free \rtimes_\Phi \Z$, Feighn--Handel \cite[Prop.\,2.3]{FH99} show that~$\G'$ has the relative presentation
\[ \langle \mb \free'', s~|~sas^{-1} = x\Phi^n(a)x^{-1} \text{ for all } a \in \mb A \rangle,\]
where~$\mb A$ is a free factor of a finitely generated subgroup $\free'' \le \free'$ such that $x\Phi^n(\mb A)x^{-1} \le \free''$.
This presentation is \emph{aspherical} and allows us to compute: $0 = \chi(\G') = \rank(\mb A) - \rank(\free'')$;
therefore, $\mb A = \free''$.
Since $\pi|\G'$ maps $\free'' \mapsto 0$ and $s \mapsto n$ and its kernel~$\free'$ is free, we have $x\Phi^n(\free'')x^{-1} = \free''$ -- otherwise, if $x\Phi^n(\free'')x^{-1} \neq \free''$, then the kernel is (locally free but) not free.
In particular, $\free' = \ker(\pi|\G') = \free''$ is finitely generated.
\end{proof}

We are now ready to give the characterisation:
\begin{prop}\label{propPolySubgp2}
$\mc P(\G)$ is the unique (up to conjugacy) set of pairwise nonconjugate maximal polynomial free-by-cyclic subgroups of~$\G$.
\end{prop}

\noindent 
As the first step of the proof, we show~$\mc P(\G)$ is \underline{malnormal} in~$\G$: essentially distinct conjugates of $\mb G_1, \mb G_2 \in \mc P(\G)$ have trivial intersection.
Another proof is given after \cref{thmRelHyp}.

\begin{proof}
We reconstruct~$\mc P(\G)$ along the same lines we constructed~$\mc P(\phi)$ in \cref{propPolySubgp}.
Suppose~$\G$ is not polynomial.
Then there is a minimal isometric $\free$-action on an $\mb R$-tree $T$ with trivial arc stabilisers, and the $\mb R$-tree $T$ admits a $\Phi$-equivariant expanding homothety.
As in the proof of \cref{lemZFreeSplit}, we can use the expanding $\Phi$-equivariant homothety to define the unique homothetic $\G$-action on~$T$ that extends the isometric $\free$-action.

Define a child $\mb H_i$ of $T$ (for the $\G$-action) to be a noncyclic point stabiliser representing a conjugacy class.
Such a child is a free-by-cyclic subgroup $\mb H_i = \mb S_i \rtimes_{\Phi_i} \Z$, where~$\mb S_i$ is child of $T$ (for the $\free$-action) representing a $\phi$-orbit of conjugacy classes in~$\free$.
So the collection~$\mc H$ of the children of~$T$ (for the $\G$-action) is finite too.
Since expanding homotheties have at most one fixed point, the $\G$-action on~$T$ still has trivial arc stabilisers.
In particular,~$\mc H$ is malnormal in~$\G$.
By inductively considering nonpolynomial children, we get a malnormal family tree whose collection of terminal descendants is~$\mc P(\G)$.
Malnormality is transitive; therefore, $\mc P(\G)$ is malnormal in~$\G$.

Let $\G' \le \G$ be a free-by-cyclic subgroup.
By \cref{lemIntersect}, $\free' = \G' \cap \free$ is finitely generated, and~$\G'$ is generated by~$\free'$ and $s \in \G$.
In particular, $\G' = \free' \rtimes_{\Psi} \Z$, where $\Psi \colon \free' \to \free'$ represents a restriction $\psi \in \out(\free')$ of $\phi \in \out(\free)$.
If~$\G'$ is polynomial, then the conjugacy class~$[x']_{\free'}$ (in~$\free'$) grows polynomially on $\psi$-iteration for all $x' \in \free'$.
This implies $[x']_\free$ grows polynomially on $\phi^n$-iteration, and hence, also on $\phi$-iteration.
By \cref{propPolySubgp}, $\free'$ is a subgroup of some  (conjugate in~$\free$ of) $\mb P_i \in \mc P(\phi)$.
Then~$s$ normalises~$\mb P_i$ by malnormality of $\mc P(\phi)$ in~$\free$, and $\G' = \langle \free', s \rangle$ is a subgroup of $\langle \mb P_i, s \rangle \le \G_i \in \mc P(\G)$.
So~$\mc P(\G)$ must be the collection of maximal polynomial free-by-cyclic subgroups of~$\G$ by malnormality in~$\G$.
\end{proof}

We end the section by introducing a canonical $\Z$-hierarchy that realises the $\Z$-depth.
Let $T, T'$ be $\Z$-splittings of~$\G$.
$T$ \underline{dominates}~$T'$ if every $T$-point stabiliser fixes a point in~$T'$.
Two $\Z$-splittings of~$\G$ are \underline{equivalent} if they dominate each other;
domination induces a partial order on the equivalence classes.
Note that equivalent $\Z$-splittings have the same children.
Analogous to {JSJ-decompositions for 3-manifolds}, Rips--Sela \cite[Thm.\,7.1]{RS97} defined a canonical equivalence class of~$\Z$-splittings:

\begin{thm}\label{thmRSJSJ}
Some $\Z$-splitting of~$\G$ dominates all $\Z$-splittings of~$\G$ if $\free$ is not cyclic. \qed
\end{thm}

If $\G \cong \Z \rtimes \Z$ (i.e.~$\free$ is cyclic), then~$\G$ has a unique $\Z$-hierarchy, which we call the \underline{canonical $\Z$-hierarchy} for~$\G$.
Otherwise, define the \underline{(canonical) $\Z$-children} of~$\G$ to be the children of the maximal equivalence class of $\Z$-splittings of~$\G$.
By inductively considering the canonical {$\Z$-descendants}, we get the \underline{canonical $\Z$-hierarchy} for~$\G$.

\begin{cor}\label{corCanonDepth}
The depth of the canonical $\Z$-hierarchy for~$\G$ is~$\delta_\Z(\G)$.
\end{cor}
\begin{proof}
Let $\delta_\mathrm{can}(\G)$ be the depth of the canonical $\Z$-hierarchy for~$\G$.
We may assume~$\free$ is not cyclic and, by \cref{thmRSJSJ},  $\delta_\Z(\G) \ge 1$.
We need to show $\delta_\mathrm{can}(\G) \le \delta_\Z(\G)$.
Pick a $\Z$-splitting~$T'$ of~$\G$ that initiates a $\Z$-hierarchy for~$\G$ with minimal depth $\delta_\Z(\G)$.
Then children of~$T'$ have $\Z$-depth $< \delta_\Z(\G)$.
By domination (\cref{thmRSJSJ}), each $\Z$-child~$\mb S_i$ of~$\G$ is (conjugate to) a subgroup of some child~$\mb S_{j_i}'$ of~$T'$.
If we can show that $\delta_\mathrm{can}(\mb S_i) \le \delta_\Z(\mb S_{j_i}')$ for each $\Z$-child~$\mb S_i$, then $\delta_\mathrm{can}(\G) \le \delta_\Z(\G)$ since $\delta_\Z(\mb S_{j_i}') < \delta_\Z(\G)$.

We proceed by induction on the $\Z$-depth.
For the base case, let $\delta_\Z(\mb S_{j_i}') = 0$.
If~$\mb S_{j_i}'$ has an absolute $\Z$-splitting~$T_{j_i}'$, then~$\mb S_i$ acting on its minimal subtree in~$T_{j_i}'$ is an absolute $\Z$-splitting and $\delta_\mathrm{can}(\mb S_i) = 0$.
If~$\mb S_{j_i}'$ does not split over~$\Z$, then it must be a subgroup of a $\Z$-child of~$\G$;
in fact, $\mb S_{j_i}' = \mb S_i$ and $\delta_\mathrm{can}(\mb S_i) = 0$.

Now assume $\delta_\Z(\mb S_{j_i}') > 0$ and: if a $\Z$-child~$\mb S$ of a free-by-cyclic group~$\G'$ is a subgroup of a child~$\mb S'$ of a $\Z$-splitting of~$\G'$ with $\delta_\Z(\mb S') < \delta_\Z(\mb S_{j_i}') $, then $\delta_\mathrm{can}(\mb S) < \delta_\Z(\mb S_{j_i}')$.
Let~$T_{j_i}'$ be a $\Z$-splitting of~$\mb S_{j_i}'$ that initiates a $\Z$-hierarchy with minimal depth $\delta_\Z(\mb S_{j_i}')$.
Then the children of~$T_{j_i}'$ have $\Z$-depth $< \delta_\Z(\mb S_{j_i}')$.
If~$\mb S_i$ is a subgroup of a child of~$T_{j_i}'$, then $\delta_\mathrm{can}(\mb S_i) < \delta_\Z(\mb S_{j_i}')$ by the induction hypothesis.
Otherwise, the minimal subtree $T'' \subset T_{j_i}'$ of~$\mb S_i$ is a $\Z$-splitting.
The $\Z$-children of~$\mb S_i$ are subgroups of children of~$T''$ (\cref{thmRSJSJ}), which are in turn subgroups of children of~$T_{j_i}'$.
By the induction hypothesis again, the $\Z$-children of~$\mb S_i$ have canonical $\Z$-hierarchies with depth $< \delta_\Z(\mb S_{j_i}')$, and $\delta_\mathrm{can}(\mb S_i) \le \delta_\Z(\mb S_{j_i}')$.
\end{proof}

\section{Geometric invariance}\label{secGeomInv}
Throughout, $\free'$ denotes a finitely generated noncyclic free group, $\Psi \colon \free' \to \free'$ an automorphism, $\psi = [\Psi]$ its outer class, and~$\G'$ the free-by-cyclic group $\free' \rtimes_\Psi \Z$.
We also assume~$\G$ is not virtually abelian, i.e.~$\free$ is not cyclic.

The {Cayley graph} $\cay(\mb H, S)$ of a group~$\mb H$ with respect to a finite generating set $S \subset \mb H$ is a graph whose 0-skeleton is~$\mb H$ and 1-cells connect $g,gs \in \mb H$ for all $(g,s) \in \mb H \times S$.
A function $f\colon \Gamma \to \Gamma'$ on connected locally finite graphs is a \underline{quasi-isometric (q.i.) embedding} if there is a constant $K \ge 1$ such that
\[ \frac{1}{K}d(p_1, p_2) - K \le d'(f(p_1),f(p_2)) \le K d(p_1, p_2) + K \text{ for all } p_1, p_2 \in \Gamma, \]
where $d, d'$ are the combinatorial metrics on $\Gamma, \Gamma'$ respectively;
$f$ is a \underline{quasi-isometry} if, additionally,~$\Gamma'$ is the $K$-neighbourhood of the image~$f(\Gamma)$.
Quasi-isometries determine an equivalence relation on connected locally finite graphs.
For any two finite generating sets $S_1, S_2 \subset \mb H$, the identity map on~$\mb H$ extends to a quasi-isometry $\cay(\mb H, S_1) \to \cay(\mb H, S_2)$;
thus, $\cay(\mb H)$ is well-defined up to quasi-isometry.
A finitely generated subgroup $\mb H' \le \mb H$ is \underline{undistorted} if the inclusion extends to a q.i.-embedding $\cay(\mb H') \to \cay(\mb H)$;
for instance, if $\mb H' \le \mb H$ has finite index, then~$\mb H'$ acts (by left multiplication) freely and cocompactly on~$\cay(\mb H)$, and the inclusion $\mb H' \le \mb H$ extends to a quasi-isometry.

Here is our second surprising property of free-by-cyclic subgroups in~$\G$:

\begin{lem}\label{lemUndistorted}
Free-by-cyclic subgroups of~$\G$ are undistorted.
\end{lem}
\begin{proof}
To prove the lemma, it suffices to show that the free-by-cyclic subgroup $\G' \le \G$ is undistorted in a finite index subgroup of~$\G$.
\Cref{lemIntersect} states that $\G' \cap \free$ is finitely generated, and we may assume $\free' = \G' \cap \free$.
In particular, $\psi \in \out(\free')$ is a restriction of an iterate of $\phi \in \out(\free)$.
After replacing~$\G$ with a finite index subgroup if necessary, we may assume~$\psi$ is a restriction of~$\phi$, and~$\G'$ is generated by~$\free'$ and~$t$ (i.e.~$\Psi = \Phi|\free'$).
By Hall's theorem,~$\free'$ is a free factor of a finite index subgroup of~$\free$ (see \cite[\S6]{Sta83});
we may replace~$\free$ with the intersection of the $\Phi$-iterates of the finite index subgroup and assume~$\free' \le \free$ is a free factor.
Pick a basis~$\mf B'$ of~$\free'$ and extend it to a basis~$\mf B$ of~$\free$.
We will show that the inclusion of $\cay(\G') = \cay(\G', \mf B'\cup\{t\})$ into $\cay(\G) = \cay(\G, \mf B\cup\{t\})$ is a q.i.-embedding by defining a map $r \colon \cay(\G) \to \cay(\G')$ that is a Lipschitz retract.

Note that $\cay(\free') = \cay(\free', \mf B')$ is a subtree of $\cay(\free) = \cay(\free, \mf B)$ and the closest point projection $s \colon \cay(\free) \to \cay(\free')$ is 1-Lipschitz.
Define~$r$ on the $t^n$-translates of $\cay(\free) \subset \cay(\G)$ by setting $r(t^n \cdot p) = t^n \cdot s(p)$ for all $p \in \cay(\free), n \in \Z$.
For each remaining 1-cell~$e$ of~$\cay(\G)$, let~$r$ map~$e$ linearly to an arbitrary shortest path in~$\cay(\G')$ connecting the $r$-images of its endpoints.
The restriction~$r|\cay(\G')$ is the identity map.

We now exhibit a uniform diameter for the $r$-image of any 1-cell in~$\cay(\free)$.
This is immediate for 1-cells in $t^n \cdot \cay(\free)$ for $n \in \Z$ since~$s$ is 1-Lipschitz.
Consider the 1-cell~$e$ connecting $t^{n-1}\Phi(x) = t^nxt^{-1}$ and~$t^nx$, where~$x \in \free$.
We need a uniform bound on the distance in~$\cay(\G')$ between $r(t^{n-1}\Phi(x)) = t^{n-1}s(\Phi(x))$ and $r(t^nx) = t^ns(x)$.
The latter is adjacent to $t^{n-1}\Phi(s(x))$, and it is enough to give a uniform bound on the distance in~$\cay(\free')$ between~$s(\Phi(x))$ and~$\Phi(s(x))$.
Set $y=\Phi^{-1}(s(\Phi(x))) \in \free'$ and consider the shortest path~$[x,y]$ in~$\cay(\free)$ between~$x,y$.
Then $s(x) \in [x,y]$ and $\Phi(y) = s(\Phi(x)) \in [\Phi(x), \Phi(s(x))]$ by definition of~$s$.
Bounded cancellation \cite[p.\,454]{Coo87} states that some uniform constant, independent of~$x \in \free$, bounds the distance in~$\cay(\free)$ between~$\Phi(s(x))$ and $[\Phi(x), s(\Phi(x))]$, which is also the distance in~$\cay(\free')$ between~$\Phi(s(x))$ and~$s(\Phi(x))$.
\end{proof}

By a very similar argument, the lemma extends to cyclic subgroups.

\begin{lem}\label{lemUndistorted2}
Slender subgroups of~$\G$ are undistorted.
\end{lem}

\noindent A variation of this appears in \cite[\S3]{Mit98}.

\begin{proof}
Any subgroup $\Z \rtimes \Z \le \G$ is undistorted by \cref{lemUndistorted}.
First, suppose $\langle c \rangle \le \G$ is not a subgroup of~$\free$. 
After replacing~$\G$ with the finite index subgroup generated by~$\free$ and~$c$, we may assume $c = t$.
Then $\langle c \rangle$ is a section of the homomorphism $\G \to \Z$ that maps $\free$ to~$0$ and~$t$ to~$1$;
thus $\langle c \rangle \le \G$ is undistorted.
Now suppose $\langle c \rangle \le \free$ is not trivial.
If the conjugacy class~$[c]$ (in~$\free$) is $\phi$-periodic, then $\langle c \rangle \le \Z^2 \le \G$.
By \cref{lemUndistorted} again, $\Z^2 \le \G$ is undistorted.
All subgroups of~$\Z^2$ are undistorted, and hence $\langle c \rangle \le \G$ is undistorted.

We may assume the conjugacy class~$[c]$ \emph{strictly grows} on $\phi^{\pm 1}$-iteration.
Fix a basis for~$\free$;
for nontrivial $x \in \free$, let $\alpha(x) \subset \cay(\free)$ denote the axis for~$x$ in the Cayley tree.
Consider $C = \bigcup_{n \in \Z} t^n \cdot \alpha(\Phi^{-n}(c))$ in $\cay(\G)$.
Make~$C$ connected by including, for each vertex $t^{n+1}x \in C$, the edge connecting it to~$t^n \Phi(x)$ and the shortest path in $t^n \cdot \cay(\free)$ from~$t^n \Phi(x)$ to~$C$;
by bounded cancellation, the length of the latter paths are uniformly bounded.
Let $s_n \colon \cay(\free) \to \alpha(\Phi^{-n}(c))$ be the closest point projections, and define $r \colon \cay(\G) \to C$ by setting $r(t^n \cdot p) = t^n \cdot s_n(p)$ for all $p \in \cay(\free), n \in \Z$ and extending linearly on the remaining edges.
As in the previous proof, $C$ is undistorted in~$\cay(\G)$ if there is a uniform bound on the distance in~$\cay(\free)$ between $s_n(\Phi(x))$ and $\Phi(s_{n+1}(x))$ for~$x \in\free$, which follows from bounded cancellation again.
Since~$[c]$ strictly grows on $\phi^{\pm 1}$-iteration,~$\alpha(c)$ is \emph{quasiconvex} in~$C$;
therefore, $\langle c \rangle \le \G$ is undistorted.
\end{proof}

A property of connected locally finite graphs is a \underline{geometric invariant} if it is preserved by quasi-isometries;
for instance, being \underline{one-ended} -- that is, the complement of any bounded set has exactly one unbounded component -- is a geometric invariant.
Our goal in this section is to relate growth type of~$\phi$ with a geometric invariant of $\cay(\G)$!

A connected subgraph $\ell \subset \cay(\G)$ is a \underline{quasi-geodesic} if there is a q.i.-embedding $q \colon \cay(\Z) \to \cay(\G)$ whose image is in~$\ell$ and has a {finite neighbourhood} (i.e.~the (closed) $M$-neighbourhood $N_M(\ell)$ for some $M \ge 0$) containing~$\ell$;
a component of $\cay(\G) \setminus \ell$ is \underline{essential} if its union with~$\ell$ is one-ended.
A quasi-geodesic~$\ell$ \underline{(strongly) separates}~$\cay(\G)$ if $\cay(\G) \setminus \ell$ has at least two essential components and 
a function $D\colon \mb R_{\ge 0} \to \mb R_{\ge 0}$ such that if~$E \ge 0$ and~$C$ is a component of $\cay(\G) \setminus N_E(\ell)$ that is not in $N_{D(E)}(\ell)$, then~$C$ is not in any finite neighbourhood of~$\ell$, and~$\ell$ is in $N_{D(E)}(C)$.
The existence of a separating quasi-geodesic is a geometric invariant:

\begin{lem}\label{lemQILines}
If $f \colon \cay(\G) \to \cay(\G')$ is a quasi-isometry and~$\ell$ a quasi-geodesic separating~$\cay(\G)$, then some finite neighbourhood of~$f(\ell)$ is a quasi-geodesic separating~$\cay(\G')$.
\end{lem}

\noindent This is stated in \cite[Lem.\,1.7]{Pap05} but with a weaker notion of separating; however, that statement with its weaker condition is false.
The stronger notion is due to Papasoglu too.
Note that the cited lemma dealt with \emph{quasi-lines}, but quasi-geodesics will do for~$\G$ thanks to \cref{lemUndistorted2}.

\begin{proof}[Sketch]
Let $K \ge 1$ be the \emph{q.i.-constant} for~$f$.
Then the $K$-neighbourhood~$N_{K}(f(\ell))$ is connected.
Let $q\colon \cay(\Z) \to \cay(\G)$ be a q.i.-embedding whose image is in~$\ell$ and has a finite neighbourhood containing~$\ell$.
Then $f \circ q$ is a q.i.-embedding whose image is in~$f(\ell)$ and has a finite neighbourhood containing $N_{K}(f(\ell))$.
So $\ell' = N_{K}(f(\ell))$ is a quasi-geodesic. 

Suppose~$\ell$ is separating with corresponding function $D \colon \mb R_{\ge 0} \to \mb R_{\ge 0}$.
For $E' \ge 0$, set $D'(E') = \max(E'K^2+6K^3+K^2+K,D(E'K+4K^2)K+2K)$.
Let~$C'$ be a component of $\cay(\G') \setminus N_{E'}(\ell')$ that is not in $N_{D'(E')}(\ell')$.
Then some component~$C''$ of $C' \setminus N_{E'K^2+7K^3+K^2+2K}(f(\ell))$ is not in $N_{D(E'K+4K^2)K+K}(f(\ell))$.
The preimage $f^{-1}(C'')$ is in a component~$C$ of $\cay(\G) \setminus N_{E'K+4K^2}(\ell)$ but not in $N_{D(E'K+4K^2)}(\ell)$.
So~$C$ is not in any finite neighbourhood of~$\ell$, and~$\ell$ is in $N_{D(E'K+4K^2)}(C)$.
Finally,~$f$ maps~$C$ into the component~$C'$ of $\cay(\G')\setminus N_{E'+K}(f(\ell))$;
thus $C'$ is not in any finite neighbourhood of~$\ell'$, and~$\ell'$ is in $N_{D'(E')}(C')$.

Let $E' \ge 0$ and~$C'$ be a component of $\cay(\G') \setminus N_{E'}(\ell')$ that is not in any finite neighbourhood of~$\ell'$.
For a contradiction, suppose that the complement in $C' \cup N_{E'}(\ell')$ of some connected compact subgraph $K \subset C' \cup N_{E'}(\ell')$ has at least two unbounded components.
Replace~$K$ with a larger connected compact subgroup such that $N_{E'}(\ell') \setminus K$ has exactly two components and they are unbounded, and $(C' \cup N_{E'}(\ell')) \setminus K$ has only unbounded components.
As $\cay(\G')$ is one-ended, each component of $(C' \cup N_{E'}(\ell')) \setminus K$ contains a component of $N_{E'}(\ell') \setminus K$.
So $(C' \cup N_{E'}(\ell')) \setminus K$ has exactly two unbounded components~$C_\pm'$, and no finite neighbourhood of~$C_+'$ or~$C_-'$ contains~$\ell'$.
Pick~$R \ge 0$ so that $N_{E'+R}(\ell')$ contains~$K$.
Without loss of generality, some component of $C_+' \setminus N_{E'+R}(\ell')$ is a component of $\cay(\G') \setminus N_{E'+R}(\ell')$ that is not in any finite neighbourhood of~$\ell'$.
By the previous paragraph,~$\ell'$ is in some finite neighbourhood of~$C_+'$ -- a contradiction.

Let $C_1 \neq C_2$ be two essential components of $\cay(\G) \setminus \ell$.
For $i = 1,2$, some component~$C_i^*$ of $C_i \setminus N_{3K^4+2K^3+5K^2}(\ell)$ is not in any finite neighbourhood of~$\ell$.
So~$f$ maps $C_i^*$ into the component~$C_i'$ of $\cay(\G') \setminus N_{3K^3+2K^2+2K}(f(\ell))$ that is not in any finite neighbourhood of~$f(\ell)$.
For a contradiction, suppose $C' = C_1' = C_2'$.
Then the preimage $f^{-1}(C')$ is in a component of $\cay(\G)\setminus N_{K}(\ell)$.
So $C_1^*, C_2^*$ are in the same component of $\cay(\G) \setminus \ell$, which contradicts $C_1 \neq C_2$;
therefore, $C_1' \neq C_2'$ and $N_{3K^3+2K^2+K}(\ell')$ is separating.
\end{proof}

The following is a deep theorem of Papasoglu \cite[Thm.\,1]{Pap05}:

\begin{thm}\label{thmZSplit}
$\cay(\G)$ has a separating quasi-geodesic if and only if~$\G$ splits over~$\Z$. \qed
\end{thm}

So splitting over~$\Z$ is a geometric invariant of~$\G$!
The cited theorem applies to more general finitely presented groups.
We have specialised and simplified all statements in this section for our free-by-cyclic group~$\G$.
We say $\G, \G'$ are \underline{quasi-isometric} if their Cayley graphs are quasi-isometric.
The following conjecture motivates an extension of \cref{thmZSplit} to slender splittings.

\begin{conj}\label{conjSlender}
If $\G, \G'$ are quasi-isometric, then a slender splitting of~$\G$ induces a slender splitting of~$\G'$.
\end{conj}

Two subsets of a metric space are a \underline{finite Hausdorff distance} apart if each subset is in the finite neighbourhood of the other.
By another theorem of Papasoglu \cite[Thm.\,7.1]{Pap05}, the $\Z$-children are geometric invariants as well:

\begin{thm}\label{thmZChild}
If $f \colon \cay(\G) \to \cay(\G')$ is a quasi-isometry and~$\G_0$ a $\Z$-child of~$\G$, then $f(\G_0)$ is a finite Hausdorff distance from a $\Z$-child of~$\G'$. \qed
\end{thm}

\begin{cor}\label{corZHierarchy}
If $\G, \G'$ are quasi-isometric, then $\delta_\Z(\G') = \delta_\Z(\G)$;
furthermore, if the canonical $\Z$-hierarchy for~$\G$ is complete, then so is the canonical $\Z$-hierarchy for~$\G'$.
\end{cor}
\begin{proof}
By \cref{thmZChild},~$\G$ has a $\Z$-child (i.e.~$\delta_\Z(\G) > 0$ by \cref{corCanonDepth}) if and only if~$\G'$ does too.
Thus the first part holds if $\delta_\Z(\G) = 0$.
Suppose $\delta_\Z(\G) > 0$ and the first part holds for free-by-cyclic groups with $\Z$-depth $< \delta_\Z(\G)$.
By \cref{thmZChild,lemUndistorted}, each $\Z$-child~$\G_i$ of~$\G$ is quasi-isometric to a $\Z$-child~$\G_{j_i}'$ of~$\G'$, and $\delta_\Z(\G_i) < \delta_\Z(\G)$.
Similarly, each $\Z$-child~$\G_j'$ of~$\G'$ is quasi-isometric to a $\Z$-child~$\G_{i_j}$ of~$\G$.
By the induction hypothesis, $\delta_\Z(\G_i) = \delta_\Z(\G_{j_i}')$, $\delta_\Z(\G_j') = \delta_\Z(\G_{i_j})$ for all $\Z$-children of~$\G, \G'$;
therefore, $\delta_\Z(\G') = \delta_\Z(\G)$.

For the second part of the theorem, assume the canonical $\Z$-hierarchy for~$\G$ is complete, and set $\delta = \delta_\Z(\G') = \delta_\Z(\G)$.
If $\delta = 0$, then both $\G, \G'$ have absolute $\Z$-splittings by \cref{thmZSplit,lemQILines}, and the canonical $\Z$-hierarchy for~$\G'$ is complete too.
Suppose $\delta > 0$ and the second part holds for free-by-cyclic groups with $\Z$-depth $< \delta$.
As above, each $\Z$-child of~$\G'$ is quasi-isometric to a $\Z$-child of~$\G$, and the latter are assumed to have complete canonical $\Z$-hierarchies.
By the induction hypothesis, each $\Z$-child of~$\G'$ has a complete canonical $\Z$-hierarchy and so does~$\G'$.
\end{proof}

By \cref{thmGroupInv,corZHierarchy}, growth type is a geometric invariant:

\begin{cor}\label{corQIInv}
If $\G, \G'$ are quasi-isometric and~$\G$ is polynomial, then so is~$\G'$. \qed
\end{cor}

\noindent This also follows from \cref{thmRelHyp} below.
Macura \cite[Thm.\,1.2]{Mac02} gave a geometric characterisation of the degree of a polynomially growing outer automorphism:

\begin{thm}\label{thmDegQIInv}
A polynomial~$\G$ has polynomial divergence with degree $\Deg(\phi)+1$. \qed
\end{thm}

Roughly speaking, the {divergence} is a (collection of) function(s) that measures how quickly geodesic rays diverge.
The interested reader should refer to Stephen Gersten's paper \cite[\S2]{Ger94}, where divergence is introduced and noted to be a geometric invariant.
Recently, Mark Hagen \cite[Thm.\,1.2]{Hag19} `modernised' Macura's theorem by showing that a polynomial~$\G$ is {strongly thick} of order $\deg(\phi)$.
{Thickness} is a more structural property (compared to divergence) that was introduced by Behrstock--Dru\c{t}u--Mosher \cite[\S7]{BDM09} as an obstruction to relative hyperbolicity.

Alternatively, one could reprove the geometric invariance of the degree by showing: $\delta_s(\G)$ is the depth of a canonical slender hierarchy for~$\G$; and this hierarchy is a geometric invariant.
Fujiwara--Papasoglu \cite[Thm.\,5.13]{FP06} already developed a slender analogue of Rips--Sela's \cref{thmRSJSJ}.
We currently have no geometric characterisation of when~$\G$ splits over $\Z \rtimes \Z$, hence our \cref{conjSlender}.


\section{Relative hyperbolicity}\label{secRelHyp}

We finally address the geometry of free-by-cyclic groups that are not polynomial.

First, consider the extreme case  when $\mc P(\G)$ is empty, or equivalently, $\mc P(\phi)$ is empty.
By \cref{propDichotomy}, polynomially growing outer automorphisms of~$\free$ have periodic conjugacy classes of nontrivial elements in~$\free$, e.g.~nontrivial elements in terminal descendants of a complete fixed free hierarchy;
thus, $\mc P(\phi)$ is empty if and only if~$\phi$ is \underline{atoroidal}: there are no $\phi$-periodic conjugacy classes of nontrivial elements in~$\free$.
We remark that $\mc P(\G)$ is empty if and only if~$\G$ has no free abelian subgroup of rank~2.
Peter Brinkmann \cite[Thm.\,1.1]{Bri00} proved that the absence of~$\Z^2$ subgroups is a geometric invariant:

\begin{thm}\label{thmHyperbolic}
The following are equivalent:
\begin{enumerate}
\item $\cay(\G)$ is hyperbolic;
\item $\G$ has no $\Z^2$ subgroups; and
\item $\phi$ is atoroidal. \qed
\end{enumerate}
\end{thm}

A locally finite graph is \underline{hyperbolic} if there is a $\delta \ge 0$ such that the $\delta$-neighbourhood of the union of two sides of any geodesic triangle in the graph contains the third side of the triangle;
Misha Gromov \cite[Cor.\,2.3.E]{Gro87} introduced this geometric invariant.
Gromov also introduced {relative hyperbolicity} in \cite[\S8.6]{Gro87} as a group property that generalises lattices of negatively curved symmetric spaces.

\Cref{thmHyperbolic} gives a geometric characterisation of when~$\mc P(\G)$ is empty.
Recently, this was generalised to the case when $\mc P(\G) \neq \{ \G\}$, i.e.~$\G$ is not polynomial.

\begin{thm}\label{thmRelHyp}
The following are equivalent:
\begin{enumerate}
\item $\G$ is hyperbolic rel.~$\mc P(\G)$;
\item $\G$ is relatively hyperbolic;
\item $\G$ is not polynomial; and
\item $\phi$ is exponentially growing.
\end{enumerate}
\end{thm}

\noindent Cornelia Dru\c{t}u \cite[Thm.\,1.2]{Dru09} proved that relative hyperbolicity is a geometric invariant!
So this theorem gives another proof of \cref{corQIInv}.

\begin{proof}[Outline]
($2{\Rightarrow}3$) is essentially Macura's \cref{thmDegQIInv} since it is folklore that relatively hyperbolic groups have exponential divergence.
Alternatively, this is Hagen's theorem since thick groups cannot be relatively hyperbolic.

($4{\Rightarrow}1$) was initially announced by Gautero--Lustig in 2007, but their proof was incomplete.
Pritam Ghosh \cite[Cor.\,3.16]{Gho23} and Dahmani--Li \cite[Thm.\,0.4]{DL22} recently gave complete independent proofs.
\end{proof}

\begin{proof}[Another proof for \cref{propPolySubgp2}]
Suppose $\G$ is not polynomial. 
Then~$\G$ is hyperbolic rel.~$\mc P(\G)$ by \cref{thmRelHyp}($3{\Rightarrow}1$).
As any free-by-cyclic subgroup is undistorted (\cref{lemUndistorted}), $\G' \le \G$ is relatively hyperbolic or conjugate into some subgroup in~$\mc P(\G)$ \cite[Thm.\,1.8]{DS05}.
By \cref{thmRelHyp}($2{\Rightarrow}3$), if $\G' \le \G$ is polynomial, then it is conjugate into some subgroup in~$\mc P(\G)$.
By malnormality of peripheral structures, $\mc P(\G)$ is the collection of maximal polynomial free-by-cyclic subgroups of~$\G$.
\end{proof}

Not only is $\mc P(\G)$ a group invariant of~$\G$, it turns out to be a geometric invariant!
Behrstock--Dru\c{t}u--Mosher \cite[Thm.\,4.8]{BDM09} strengthened Dru\c{t}u's geometric invariance theorem when the group is hyperbolic relative to non-(relatively hyperbolic) subgroups:

\begin{thm}If $f \colon \cay(\G) \to \cay(\G')$ is a quasi-isometry and $\mb P \in \mc P(\G)$, then $f(\mb P)$ is a finite Hausdorff distance from a conjugate of some $\mb P' \in \mc P(\G')$. \qed 
\end{thm}

To conclude, we return to the extreme case when~$\phi$ is atoroidal.
The outer automorphism~$\phi$ is \underline{fully irreducible} if there are no $\phi$-periodic conjugacy classes of nontrivial proper free factors of~$\free$.
Using \cref{lemIntersect}, we \cite[Thm.\,4.3]{Mut21} proved that being both fully irreducible and atoroidal was a group (actually commensurability) invariant:

\begin{thm}
$\G$ has no infinite index free-by-cyclic subgroups if and only if~$\phi$ is fully irreducible and atoroidal. \qed
\end{thm}

In \cite[p.\,48]{Mut21}, we conjecture that being fully irreducible and atoroidal is a geometric invariant, but it is not clear what the equivalent property of the Cayley graph would be.
In fact, we suspect something more general!
Recall the definition of {(lamination) depth}~$\mf d(\phi)$: the length of the longest properly nested sequence(s) of attracting laminations for~$\phi$.

\begin{conj}\label{conjHeight}
If $\G, \G'$ are quasi-isometric, then $\mf d(\phi) = \mf d(\psi)$.
\end{conj}

The depth 0 case of the conjecture is precisely \cref{corQIInv}.
For atoroidal outer automorphisms, being fully irreducible may be equivalent to having depth 1 and no fixed free splittings;
thus, the depth 1 case of the conjecture may be equivalent to the geometric invariance of being fully irreducible and atoroidal.



\bibliography{zrefs}

\begin{thebibliography}{BDM09}

\bibitem[BDM09]{BDM09}
Jason Behrstock, Cornelia Dru\c{t}u, and Lee Mosher.
\newblock Thick metric spaces, relative hyperbolicity, and quasi-isometric
  rigidity.
\newblock {\em Math. Ann.}, 344(3):543--595, 2009.

\bibitem[BFH00]{BFHTits1}
Mladen Bestvina, Mark Feighn, and Michael Handel.
\newblock The {T}its alternative for {${\rm Out}(F_n)$}. {I}. {D}ynamics of
  exponentially-growing automorphisms.
\newblock {\em Ann. of Math. (2)}, 151(2):517--623, 2000.

\bibitem[BFH05]{BFHTits2}
Mladen Bestvina, Mark Feighn, and Michael Handel.
\newblock The {T}its alternative for {${\rm Out}(F_n)$}. {II}. {A} {K}olchin
  type theorem.
\newblock {\em Ann. of Math. (2)}, 161(1):1--59, 2005.

\bibitem[BH92]{BH92}
Mladen Bestvina and Michael Handel.
\newblock Train tracks and automorphisms of free groups.
\newblock {\em Ann. of Math. (2)}, 135(1):1--51, 1992.

\bibitem[Bri00]{Bri00}
Peter Brinkmann.
\newblock Hyperbolic automorphisms of free groups.
\newblock {\em Geom. Funct. Anal.}, 10(5):1071--1089, 2000.

\bibitem[Bri02]{Bri02}
Peter Brinkmann.
\newblock Splittings of mapping tori of free group automorphisms.
\newblock {\em Geom. Dedicata}, 93:191--203, 2002.

\bibitem[Coo87]{Coo87}
Daryl Cooper.
\newblock Automorphisms of free groups have finitely generated fixed point
  sets.
\newblock {\em J. Algebra}, 111(2):453--456, 1987.

\bibitem[Cul84]{Cul84}
Marc Culler.
\newblock Finite groups of outer automorphisms of a free group.
\newblock In {\em Contributions to group theory}, volume~33 of {\em Contemp.
  Math.}, pages 197--207. Amer. Math. Soc., Providence, RI, 1984.

\bibitem[DL22]{DL22}
Fran\c{c}ois Dahmani and Ruoyu Li.
\newblock Relative hyperbolicity for automorphisms of free products and free
  groups.
\newblock {\em J. Topol. Anal.}, 14(1):55--92, 2022.

\bibitem[Dru09]{Dru09}
Cornelia Dru\c{t}u.
\newblock Relatively hyperbolic groups: geometry and quasi-isometric
  invariance.
\newblock {\em Comment. Math. Helv.}, 84(3):503--546, 2009.

\bibitem[DS00]{DS00}
Martin~J. Dunwoody and Eric~L. Swenson.
\newblock The algebraic torus theorem.
\newblock {\em Invent. Math.}, 140(3):605--637, 2000.

\bibitem[DS05]{DS05}
Cornelia Dru\c{t}u and Mark Sapir.
\newblock Tree-graded spaces and asymptotic cones of groups.
\newblock {\em Topology}, 44(5):959--1058, 2005.
\newblock With an appendix by Denis Osin and Mark Sapir.

\bibitem[FH99]{FH99}
Mark Feighn and Michael Handel.
\newblock Mapping tori of free group automorphisms are coherent.
\newblock {\em Ann. of Math. (2)}, 149(3):1061--1077, 1999.

\bibitem[FP06]{FP06}
Koji Fujiwara and Panos Papasoglu.
\newblock J{SJ}-decompositions of finitely presented groups and complexes of
  groups.
\newblock {\em Geom. Funct. Anal.}, 16(1):70--125, 2006.

\bibitem[Ger94]{Ger94}
Stephen~M. Gersten.
\newblock {Quadratic divergence of geodesics in {${\rm CAT}(0)$} spaces}.
\newblock {\em Geom. Funct. Anal.}, 4(1):37--51, 1994.

\bibitem[Gho23]{Gho23}
Pritam Ghosh.
\newblock Relative hyperbolicity of free-by-cyclic extensions.
\newblock {\em Compos. Math.}, 159(1):153--183, 2023.

\bibitem[GL95]{GL95}
Damien Gaboriau and Gilbert Levitt.
\newblock {The rank of actions on {${\bf R}$}-trees}.
\newblock {\em Ann. Sci. \'{E}cole Norm. Sup. (4)}, 28(5):549--570, 1995.

\bibitem[Gro87]{Gro87}
Misha Gromov.
\newblock Hyperbolic groups.
\newblock In {\em Essays in group theory}, volume~8 of {\em Math. Sci. Res.
  Inst. Publ.}, pages 75--263. Springer, New York, 1987.

\bibitem[Hag19]{Hag19}
Mark Hagen.
\newblock A remark on thickness of free-by-cyclic groups.
\newblock {\em Illinois J. Math.}, 63(4):633--643, 2019.

\bibitem[HW10]{HW10}
Mark Hagen and Daniel~T. Wise.
\newblock Special groups with an elementary hierarchy are virtually
  free-by-{$\mathbb Z$}.
\newblock {\em Groups Geom. Dyn.}, 4(3):597--603, 2010.

\bibitem[KK00]{KK00}
Michael Kapovich and Bruce Kleiner.
\newblock Hyperbolic groups with low-dimensional boundary.
\newblock {\em Ann. Sci. \'{E}cole Norm. Sup. (4)}, 33(5):647--669, 2000.

\bibitem[KL97]{KL97}
Michael Kapovich and Bernhard Leeb.
\newblock Quasi-isometries preserve the geometric decomposition of {H}aken
  manifolds.
\newblock {\em Invent. Math.}, 128(2):393--416, 1997.

\bibitem[Lev09]{Lev09}
Gilbert Levitt.
\newblock Counting growth types of automorphisms of free groups.
\newblock {\em Geom. Funct. Anal.}, 19(4):1119--1146, 2009.

\bibitem[Mac02]{Mac02}
Nata\v{s}a Macura.
\newblock Detour functions and quasi-isometries.
\newblock {\em Q. J. Math.}, 53(2):207--239, 2002.

\bibitem[Mit98]{Mit98}
Mahan Mitra.
\newblock Cannon-{T}hurston maps for hyperbolic group extensions.
\newblock {\em Topology}, 37(3):527--538, 1998.

\bibitem[Mut21]{Mut21}
Jean~Pierre Mutanguha.
\newblock Irreducibility of a free group endomorphism is a mapping torus
  invariant.
\newblock {\em Comment. Math. Helv.}, 96(1):47--63, 2021.

\bibitem[Pap05]{Pap05}
Panos Papasoglu.
\newblock Quasi-isometry invariance of group splittings.
\newblock {\em Ann. of Math. (2)}, 161(2):759--830, 2005.

\bibitem[RS97]{RS97}
Eliyahu Rips and Zlil Sela.
\newblock Cyclic splittings of finitely presented groups and the canonical
  {JSJ} decomposition.
\newblock {\em Ann. of Math. (2)}, 146(1):53--109, 1997.

\bibitem[Ser77]{Ser77}
Jean-Pierre Serre.
\newblock {\em {Arbres, amalgames, {$ {\rm SL}\sb{2}$}}}.
\newblock Soci\'{e}t\'{e} Math\'{e}matique de France, Paris, 1977.
\newblock R\'{e}dig\'{e} avec la collaboration de Hyman Bass, Ast\'{e}risque,
  No. 46.

\bibitem[Sta83]{Sta83}
John~R. Stallings.
\newblock Topology of finite graphs.
\newblock {\em Invent. Math.}, 71(3):551--565, 1983.

\end{thebibliography}
\bibliographystyle{alpha}

\end{document}